\documentclass[12pt,reqno]{amsart}

\usepackage{amsmath,amsthm,amssymb,comment,fullpage}
\usepackage{times}
\usepackage[T1]{fontenc}
\usepackage{mathrsfs}
\usepackage{latexsym}
\usepackage[dvips]{graphics}
\usepackage{epsfig}
\usepackage{amsmath,amsfonts,amsthm,amssymb,amscd}
\input amssym.def
\input amssym.tex
\usepackage{color}
\usepackage{hyperref}
\usepackage{url}
\usepackage{breakurl}
\usepackage{comment}

\newcommand{\prob}[1]{{\rm Prob}\left(#1\right)}
\newcommand{\p}[1]{{\rm Prob}\left(#1\right)}
\newcommand{\E}{\mathbb{E}}

\newcommand{\bburl}[1]{\textcolor{blue}{\url{#1}}}

\newcommand{\V}{\text{{\rm Var}}}
\renewcommand{\E}{\mathbb{E}}

\numberwithin{equation}{section}

\newtheorem{thm}{Theorem}[section]

\newtheorem{lem}[thm]{Lemma}

\newtheorem{exa}[thm]{Example}
\newtheorem{defi}[thm]{Definition}

\theoremstyle{plain}

\newcommand\be{\begin{equation}}
\newcommand\ee{\end{equation}}
\newcommand\bea{\begin{eqnarray}}
\newcommand\eea{\end{eqnarray}}
\newcommand\bi{\begin{itemize}}
\newcommand\ei{\end{itemize}}
\newcommand\ben{\begin{enumerate}}
\newcommand\een{\end{enumerate}}
\newcommand\bc{\begin{center}}
\newcommand\ec{\end{center}}
\newcommand\ba{\begin{array}}
\newcommand\ea{\end{array}}

\renewcommand{\l}{\ell}

\newcommand{\e}{\epsilon}

\newcommand{\Z}{\ensuremath{\mathbb{Z}}}
\renewcommand{\bf}{\textbf}

\newcommand{\N}{\mathbb{N}}

\newcommand\frakfamily{\usefont{U}{yfrak}{m}{n}}
\DeclareTextFontCommand{\textfrak}{\frakfamily}

\newcommand{\twocase}[5]{#1 \begin{cases} #2 & \text{{\rm #3}}\\ #4 &\text{{\rm #5}} \end{cases}}

\newcommand{\hr}[1]{\href{#1}{\url{#1}}}

\title{Benford Behavior of Generalized Zeckendorf Decompositions}

\author{Andrew Best}
\email{\textcolor{blue}{\href{mailto:ajb5@williams.edu)}{ajb5@williams.edu}}}
\address{Department of Mathematics and Statistics, Williams College, Williamstown, MA 01267}

\author{Patrick Dynes}
\email{\textcolor{blue}{\href{mailto:pdynes@clemson.edu)}{pdynes@clemson.edu}}}
\address{Department of Mathematical Sciences, Clemson University, Clemson, SC 29634}

\author{Xixi Edelsbrunner}
\email{\textcolor{blue}{\href{mailto:xe1@williams.edu}{xe1@williams.edu}}}
\address{Department of Mathematics and Statistics, Williams College, Williamstown, MA 01267}

\author{Brian McDonald}
\email{\textcolor{blue}{\href{mailto:bmcdon11@u.rochester.edu}{bmcdon11@u.rochester.edu}}}
\address{Department of Mathematics, University of Rochester, Rochester, NY 14627}

\author{Steven J. Miller}
\email{\textcolor{blue}{\href{mailto:sjm1@williams.edu, Steven.Miller.MC.96@aya.yale.edu}{sjm1@williams.edu,Steven.Miller.MC.96@aya.yale.edu} }}
\address{Department of Mathematics and Statistics, Williams College, Williamstown, MA 01267}

\author{Kimsy Tor}
\email{\textcolor{blue}{\href{mailto:ktor.student@manhattan.edu}{ktor.student@manhattan.edu}}}
\address{Department of Mathematics, Manhattan College, Riverdale, NY 10471}

\author{Caroline Turnage-Butterbaugh}
\email{\textcolor{blue}{\href{mailto:cturnagebutterbaugh@gmail.com}{cturnagebutterbaugh@gmail.com}}}
\address{Department of Mathematics, North Dakota State University, Fargo, ND 58108}

\author{Madeleine Weinstein}
\email{\textcolor{blue}{\href{mailto:mweinstein@hmc.edu}{mweinstein@hmc.edu}}}
\address{Department of Mathematics, Harvey Mudd College, Claremont, CA 91711 }

\thanks{This research was conducted as part of the 2014 SMALL REU program at Williams College and was supported by NSF grants DMS 1347804 and DMS 1265673, Williams College, and the Clare Boothe Luce Program of the Henry Luce Foundation. It is a pleasure to thank them for  their support, and the participants at SMALL and at the 16\textsuperscript{th} International Conference on Fibonacci Numbers and their Applications for helpful discussions.}



\date{\today}

\begin{document}

\begin{abstract} We prove connections between Zeckendorf decompositions and Benford's law. Recall that if we define the Fibonacci numbers by $F_1 = 1, F_2 = 2$ and $F_{n+1} = F_n  + F_{n-1}$, every positive integer can be written uniquely as a sum of non-adjacent elements of this sequence; this is called the Zeckendorf decomposition, and similar unique decompositions exist for sequences arising from recurrence relations of the form $G_{n+1}=c_1G_n+\cdots+c_LG_{n+1-L}$ with $c_i$ positive and some other restrictions. Additionally, a set $S \subset \Z$ is said to satisfy Benford's law base 10 if the density of the elements in $S$ with leading digit $d$ is $\log_{10}{(1+\frac{1}{d})}$; in other words, smaller leading digits are more likely to occur. We prove that as $n\to\infty$ for a randomly selected integer $m$ in $[0, G_{n+1})$ the distribution of the leading digits of the summands in its generalized Zeckendorf decomposition converges to Benford's law almost surely. Our results hold more generally: one obtains similar theorems to those regarding the distribution of leading digits when considering how often values in sets with density are attained in the summands in the decompositions.
\end{abstract}

\maketitle
\tableofcontents

\section{Introduction}

Zeckendorf's theorem states that every positive integer $m$ can be written uniquely as a sum of nonconsecutive Fibonacci numbers, where the Fibonacci numbers are defined by $F_{n+1} = F_n+ F_{n-1}$ with $F_1=1$ and $F_2=2$ (we must re-index the Fibonaccis, as if we included 0 or had two 1's we clearly could not have uniqueness). Such a sum is called the Zeckendorf decomposition of $m$. Zeckendorf decompositions have been generalized to many other sequences, specifically those arising from positive linear recurrences. More generally, we can consider a positive linear recurrence sequence given by \be G_{n+1}\ =\ c_1G_n + \cdots + c_LG_{n+1-L}, \ee with $c_i$ nonnegative, $L, c_1$ and $c_L$ positive, as well as rules to specify the first $L$ terms of the sequence and a generalization of the non-adjacency constraint to what is a legal decomposition. Unique decompositions exist both here and for other sequences; see \cite{Al, Day, DDKMMV, DDKMV, DG, FGNPT, GT, GTNP, Ke, KKMW, Len, MW1, MW2, Ste1, Ste2, Ze} for a sample of the vast literature on this topic.

The purpose of this paper is to connect generalized Zeckendorf decompositions and Benford's law. First discovered by Newcomb \cite{New} in the 1880s, it was rediscovered by Benford \cite{Ben} approximately fifty years later, who noticed that the distributions of the leading digits of numbers in many data sets were not uniform. In fact, there was a strong bias towards lower values. For example, the leading digit $1$ appeared about $30\%$ of the time and the leading digit $9$ under $5\%$ of the time. Data sets with such leading digit distributions are said to follow Benford's law. More precisely, the probability of a first digit base $B$ of $d$ is $\log_B(1+1/d)$, or more generally the probability that the significand\footnote{If $x>0$ and $B>1$ we may uniquely write $x$ as $S_B(x) \cdot B^{k_B(x)}$, where $S_B(x) \in [1,B)$ is the significand of $x$ and $k_B(x)$ is an integer.} is at most $s$ is $\log_B(s)$. Benford's law appears in astoundingly many data sets, from physical constants to census information to financial and behavioral data, and has a variety of applications (two of the most interesting being its use to detect accounting or voting fraud). This digit bias is in fact quite natural once one realizes that a data set will follow Benford's law if its logarithms modulo 1 are equidistributed.\footnote{Given a data set $\{x_n\}$, let $y_n = \log_{10} x_n \bmod 1$. If $\{y_n\}$ is equidistributed modulo 1 then in the limit the percentage of the time it is in $[\alpha, \beta] \subset [0,1]$ is just $\beta-\alpha$. Letting $\alpha = \log_{10} d$ and $\beta = \log_{10} (d+1)$ implies that the significand of $x$ is $d$ is $\log_{10}(1 + 1/d)$. } See \cite{Hi1,Hi2,MT-B,Rai} for more on the theory of Benford's law, as well as the edited volume \cite{M} for a compilation of articles on its theory and applications.

Before exploring whether or not the summands in Zeckendorf decompositions obey Benford's law, it's natural to ask the question about the sequence of Fibonacci numbers themselves. The answer is yes, and follows almost immediately from Binet's formula, \be F_n \ = \ \frac{5+\sqrt{5}}{10}\left(\frac{1+\sqrt{5}}{2}\right)^n +  \frac{5-\sqrt{5}}{10}\left(\frac{1-\sqrt{5}}{2}\right)^n \ee (note this is slightly different than the standard expression for Binet's formula as we have re-indexed our sequence so that the Fibonaccis begin 1, 2, 3, 5). The proof is completed by showing the logarithms modulo 1 are equidistributed, which is immediate from the irrationality of $\log_{10}(\frac{1+\sqrt{5}}{2})$ and Kronecker's theorem (if $\alpha$ is irrational then $n\alpha$ is equidistributed modulo 1) and simple book-keeping to bound the error of the secondary piece; see \cite{DG, MT-B, Was} for details.


Instead of studying Benfordness of summands in Zeckendorf decompositions, we could instead look at other properties of the summands, such as how often we have an even number or how often they are a square modulo $B$ for some fixed $B$. So long as our sequence has a positive density, our arguments will be applicable.\footnote{For example, in the limit one-third of the Fibonacci numbers are even. To see this we look at the sequence modulo 2 and find it is $1, 0, 1, 1, 0, 1, 1, 0, 1, \dots$; it is thus periodic with period 3 and one-third of the numbers are even.} We quickly review this notion. Given a set of positive integers $\mathcal{G} = \{G_n\}_{n=1}^\infty$ and a subset $S \subset \mathcal{G}$, we let $q(S, n)$ be the fraction of elements of $\mathcal{G}$ with index at most $n$ that are in  $S$:
\begin{align}
q(S,n) \ := \  \frac{\#\{G_i \in S \ : \ 1 \leq i \leq n\}}{n}.
\end{align}
When $\lim_{n \rightarrow \infty}{q(S,n)}$ exists, we define the \textbf{asymptotic density} $q(S)$ as
\begin{align}
q(S) \ := \ \lim_{n \rightarrow \infty}{q(S,n)},
\end{align} and for brevity often say the sequence $S$ has \textbf{density} $q(S)$.

In an earlier work we proved that if a set $S$ has a positive density $d$ in the Fibonaccis, then so too do its summands in the Zeckendorf decompositions, and in particular Zeckendorf decompositions using Fibonacci numbers follow Benford's law \cite{BDEMMTTW}. Our main result below is generalizing these results to the case of a \textbf{positive linear recurrence sequence}, which is a sequence of positive integers $\{G_n\}_{n=1}^{\infty}$ and a set of non-negative coefficients $c_1, \dots, c_L$ with $c_1, c_L > 0$,
\begin{align}\label{eq:recGn}
G_{n+1} \ = \ c_1G_{n}+c_1G_{n-1}+\cdots +c_{L}G_{n+1-L},
\end{align}
and prescribed positive initial terms $G_1, G_2, \dots, G_{L}$.

\begin{thm}\label{thm:mainresult} Fix a positive recurrence $\{G_n\}$.
Let $S\subseteq \{G_n\}_{n=1}^{\infty}$ be a set with positive density $d$, and fix an $\epsilon > 0$. As $n\to\infty$, for an integer $m$ selected uniformly at random from $[0,G_{n+1})$ the proportion of the summands in $m$'s Zeckendorf decomposition which belong to $S$ is within $\e$ of $d$ with probability $1+o(1)$.
\end{thm}

We define some concepts needed to prove Theorem \ref{thm:mainresult} in \S\ref{sec:preliminaries}, in particular the notion of a super-legal decomposition. We derive some needed properties of these decompositions, and then prove our main result in \S\ref{sec:proofmainthm} by showing related random variables (the number of summands, and the number of summands in our set with positive density in our recurrence sequence) are strongly concentrated.

\section{Legal and Super-Legal Decompositions}\label{sec:preliminaries}

\textbf{\emph{For the rest of the paper any positive linear recurrence sequence $\{G_n\}_{n=1}^\infty$ satisfies \eqref{eq:recGn} with $c_i \ge 0$ and $L, c_1, c_L \ge 1$.}} \\ \

Let $\{G_n\}_{n=1}^{\infty}$ be a positive linear recurrence sequence. Its the characteristic polynomial is
\begin{align}
f(\lambda)\ = \ c_0\lambda^L+c_1\lambda^{L-1}+c_2\lambda^{L-2}+\cdots+c_{L-1},
\end{align}
with roots $\lambda_1, \dots, \lambda_L$. By the Generalized Binet Formula (for a proof see, for example, Appendix A of \cite{BBGILMT}) we have $\lambda_1$ is the unique positive root, $\lambda_1 > |\lambda_2| \ge \cdots \ge |\lambda_L|$, and there exists an $A > 0$ such that
\begin{align}
G_n\ = \ A\lambda_1^n+O(n^{L-2}\lambda_2^n).
\end{align}

We introduce a few important terms needed to state our results. 

\begin{defi}\label{defn:Legal} Let $\{G_n\}$ be a positive linear recurrence sequence. Given non-negative integers $a_1, \dots, a_n$, the sum $\sum_{i=1}^n{a_iG_{n+1-i}}$ is a \textbf{legal} Zeckendorf decomposition if one of the following conditions holds.
\begin{enumerate}
\item We have $n<L$ and $a_i=c_i$ for $1\leq i\leq n$.
\item There exists an $s\in \{1,\dots,L\}$ such that
\begin{align}
a_1=c_1, \ \ \ a_2=c_2,\ \ \ \cdots \ \ \ , \ \ \ a_{s-1}=c_{s-1},\ \ \ \text{and}\ \ \ a_s<c_s,
\end{align}
$a_{s+1}, \dots, a_{s+\l}=0$ for some $\l\geq 0$, and $\{b_i\}_{i=1}^{n-s-\l}$ with $b_i=a_{s+\l+i}$ is either legal or empty.
\end{enumerate}
\end{defi}

\begin{defi}\label{defn:SuperLegal} Let $\{G_n\}$ be a positive linear recurrence sequence. Given non-negative integers $a_1, \dots, a_n$, the sum $\sum_{i=1}^n{a_iG_{n+1-i}}$ is a \textbf{super-legal} Zeckendorf decomposition if there exists an $s\in \{1, \dots, L\}$ such that
\begin{align}
a_1=c_1,\ \ \ a_2=c_2,\ \ \ \cdots \ \ \ ,\ \ \  a_{s-1}=c_{s-1}, \ \ \ \text{and}\ \ \ a_s<c_s,
\end{align}
$a_{s+1}, \dots, a_{s+\l}=0$ for some $\l\geq 0$, and $\{b_i\}_{i=1}^{n-s-\l}$ with $b_i=a_{s+\l+i}$ is either super-legal or empty.
\end{defi}
In other words, a decomposition is super-legal if it satisfies condition (2) of Definition \ref{defn:Legal}.

\begin{defi} Let $\{G_n\}$ be a positive linear recurrence sequence, and assume that the sum $\sum_{i=1}^n{a_iG_{n+1-i}}$ is a legal Zeckendorf decomposition. We call each string described by one of the conditions of Definition \ref{defn:Legal} (not counting the additional $0$'s) a \textbf{block}, and call the number of terms in each block its \bf{length}.
\end{defi}

We note that every super-legal Zeckendorf decomposition is legal and that a concatenation of super-legal Zeckendorf decompositions makes a super-legal Zeckendorf decomposition.

\begin{exa}\label{exe:recurrenceexample} The recurrence $G_{n+1}=G_n+2G_{n-1}+3G_{n-2}$ with $G_0=G_1=1, G_2=3$ produces the sequence $1,3,8,17,42,100,235,561,\dots$. The decomposition of $1274$ is
\be
1274 \ = \ 561+2(235)+2(100)+42+1\ = \ G_8+2G_7+2G_6+G_5+G_1.
\ee

This example gives coefficients $(1,2,2,1,0,0,0,1)$, so the blocks of $1274$ are $(1,2,2), \ (1,0)$, \ and $(1)$, with lengths $3$, $2$, and $1$ respectively.  Note that even though the definition of a block  excludes the additional $0$'s (i.e., the $a_{s+1}=a_{s+2}=\cdots=a_{s+\l}=0$ from the Definition \ref{defn:Legal}), it is still permissible for a block to end with a $0$. The decomposition for $1274$ is legal but not super-legal, since the final block $(1)$ satisfies condition (1) but not condition (2) from Definition \ref{defn:Legal}. 

\end{exa}

\begin{exa} An example of a super-legal decomposition using the recurrence from Example \ref{exe:recurrenceexample} is
	\be
	1277 \ =  \ 561+2(235)+2(100)+42+3+1\ = \ G_8+2G_7+2G_6+G_5+G_2+G_1,
	\ee
	which gives coefficients $(1,2,2,1,0,0,1,1)$. In this case, the final block is $(1,1)$, which satisfies the condition of Definition \ref{defn:SuperLegal}.
\end{exa}


Given two legal decompositions, we do not necessarily obtain a new legal sequence by concatenating the coefficients. However, if we require that the leading block be super-legal, we do obtain a new legal decomposition by concatenation. With the help of a few lemmas which help us count the number of super-legal decompositions, we can circumvent this obstruction.

\begin{lem}\label{eq:legal}  Let $\{G_n\}$ be a positive linear recurrence sequence with relation given by \eqref{eq:recGn}, and let $H_n$ be the number of super-legal decompositions using only $G_1,G_2,\dots,G_{n}$. We have
\begin{align}
H_{n+1} \ = \ c_1H_n+c_2H_{n-1}+\cdots+ c_LH_{n+1-L}.
\end{align}
\end{lem}

\begin{proof}
Note that $H_{n+1}-H_{n}$ is the number of super-legal decompositions with largest element $G_{n+1}$.  We count how many such decompositions there are by summing over the possible lengths of the leading block.  Say the leading block is of length $j$ with $1< j \leq L$.  Then the leading block is $(c_1,c_2,\dots,c_{j-1},a_{j})$, where $a_{j}$ is chosen from $\{0,1,\dots,c_{j}-1\}$.  Therefore there are $c_{j}$ ways of choosing this leading block. Because we require $G_{n+1}$ to be included in the decomposition, if $j=1$ there are $c_1-1$ ways of choosing this leading block, since the leading coefficient must be nonzero.  For any choice of leading block of length $j$, there are $H_{n+1-j}$ ways of choosing the remaining coefficients.  Therefore, we find that
\begin{align}
H_{n+1}-H_n \ = \ \sum_{j=1}^{L}{c_jH_{n+1-j}-H_n},
\end{align}
completing the proof.
\end{proof}

\begin{lem}\label{ratio} Let $\{G_n\}$ be a positive linear recurrence sequence, and let $H_n$ be the number of super-legal decompositions using only $G_1,G_2,\dots,G_{n}$. Then $\lim_{n\to\infty} H_n / G_n$ exists and is positive.
\end{lem}

\begin{proof}
Since $H_n$ is generated by the same recursion as $G_n$, it has the same characteristic polynomial, which then has the same roots.  Therefore for some $B\geq 0$ we have
\begin{align}
H_n\ = \ B\lambda_1^n+O(n^{L-2}\lambda_2^n).
\end{align}
Thus $\lim_{n\to\infty} H_n / G_n = B/A$ and it suffices to show that $B>0$.  Note that we always have $H_j>0$, so we must have
\begin{align}
\alpha \ := \ \min_{1\leq j\leq L}{\frac{H_j}{G_j}} \ > \ 0.
\end{align}
It follows by induction on $n$ that $H_n\geq \alpha G_n$ for all $n$. Thus we conclude that $B>0$, as desired.
\end{proof}


\section{Density Theorem}\label{sec:proofmainthm}


To prove the main result as stated in Theorem \ref{thm:mainresult}, we compute expected values and variances of the relevant random variables.  An essential part of the ensuing analysis is the following estimate on the probability that $a_j=k$ for a fixed $k$, and showing that it has little dependence on $j$.  We prove the theorem via casework based on the structure of the blocks in the decomposition of $m$. Namely, in the case that $a_j$ is in the $r$th position of a block of length $\l$, the two subcases are $r=\l$ (that is, $a_j$ is the last element in the block) or $r<\l$ (that is, $a_j$ is not the last element in the block).  This is why the notion of a super-legal decomposition is useful; if we want to know whether the legal decomposition $(a_1,a_2,\dots,a_n)$ has a block that terminates at $a_r$, this is equivalent to whether $(a_1,a_2,\dots,a_r)$ forms a super-legal decomposition. So, we first prove some useful lemmas and then collect our results to prove Theorem \ref{thm:mainresult}.

\begin{lem}\label{pk} Let $\{G_n\}$ be a positive linear recurrence sequence, and choose an integer $m$ uniformly at random from $[0,G_{n+1})$, with  legal decomposition
\begin{align}
m \ = \ \sum_{j=1}^{n}{a_jG_{n+1-j}}.
\end{align}
Note that this defines random variables $A_1, \dots, A_n$ taking on values $a_1,\dots,a_n$.

Let $p_{j,k}(n) := \p{A_j=k}.$ Then, for $\log{n}<j<n-\log{n}$, we have
\begin{align}
p_{j,k}(n) \ = \ p_k(n)(1+o(1)),
\end{align}
where $p_k(n)$ is computable and does not depend on $j$.
\end{lem}

\begin{proof}
We divide the argument into cases based on the length of the block containing $a_j$, as well as the position $a_j$ takes in this block.  Suppose that $a_j$ is in the $r$th place in a block of length $\l$.  In order to have $a_j=k$, we must either have $r<\l$ and $k=c_r$, or $r=\l$ and $k<c_r$.

In the former case, there are $c_{\l}$ ways to choose the terms in the block containing $a_j$, due to the $c_{\l}$ choices there are for the final term, and everything else is fixed.  There are $H_{j-r}$ ways to choose the coefficients for the terms greater than those in the block containing $a_j$, and $G_{n-j-\l+r+1}$ ways to choose the smaller terms.

We now consider the latter case, where $r=\l$ and $k<c_r$.  There is now only one possibility for the coefficients in the block containing $a_j$, but the rest of the argument remains the same as in the previous case.  Therefore, by Lemma \ref{ratio} we find that
\begin{align}\label{eq:expansionNjklr}
N_{j,k,\l,r}(n) & \ := \ \#\{m\in \mathbb{Z}\cap [0,G_{n+1}):a_j=k, \ a_j \ \text{is in the} \ r^{th} \text{ position in a block of length} \ \l\} \nonumber \\
& \ = \ \left\{\begin{array}{ll}
c_{\l}G_{n-j-\l+r+1}H_{j-r} & \text{ if }r<\l, \ k=c_r, \\
G_{n-j-\l+r+1}H_{j-r} & \text{ if }r=\l, \ k<c_r, \\
0 & \text{ otherwise }
\end{array}\right. \nonumber \\
& \ = \ N_{k,\l,r}(n)(1+o(1)),
\end{align}
where
\begin{align}
N_{k,\l,r}(n) \ := \ \left\{\begin{array}{ll}
c_{\l}AB\lambda_1^{n-\l+1} & \text{ if }r<\l, \ k=c_r, \\
AB\lambda_1^{n-\l+1} & \text{ if }r=\l, \ k<c_r, \\
0 & \text{ otherwise, }
\end{array}\right.
\end{align}
and $N_{k,\l,r}(n)$ does not depend on $j$; these formulas follow from applications of the Generalized Binet Formula to the sequences for the $G_n$'s and $H_n$'s.  We conclude the proof by noting that
\begin{align}
p_{j,k}(n) \ = \ \frac{1}{G_{n+1}}\sum_{\l=1}^{L}{\sum_{r=1}^{\l}{N_{j,k,\l,r}(n)}} \ = \ \left(\frac{1}{G_{n+1}}\sum_{\l=1}^{L}{\sum_{r=1}^{\l}{N_{k,\l,r}(n)}}\right) \cdot \left(1 + o(1)\right),
\end{align} where we used \eqref{eq:expansionNjklr} to replace $N_{j,k,\l,r}(n)$.  The claim now follows by defining
\begin{align}
p_k(n) \ := \ \frac{1}{G_{n+1}}\sum_{\l=1}^{L}{\sum_{r=1}^{\l}{N_{k,\l,r}(n)}}
\end{align} and noting that its size is independent of $j$. More is true, as the Generalized Binet Formula again gives us that $G_{n+1}$ is a constant times $\lambda_1^{n+1}$ (up to lower order terms), and similarly  the sum for $p_k(n)$ is a multiple of $\lambda_1^{n+1}$ plus lower order terms.
\end{proof}


We also use the following result, which follows immediately from Theorems 1.2 and 1.3 in \cite{MW1} (see also \cite{MW2} for a survey on the subject). 

\begin{lem} \label{lem:VarXn} Let $\{G_n\}$ be a positive linear recurrence sequence, with $s(m)$ the number of summands in the decomposition of $m$.  That is, for $m=\sum_{j=1}^n{a_jG_{n+1-j}}$, let $s(m):=\sum_{j=1}^n{a_j}$.  Let $X_n(m)$ be the random variable defined by $X_n(m)=s(m)$, where $m$ is chosen uniformly at random from $[0,G_{n+1})$. Then 
\begin{align}
\E[X_n(m)] \ = \  nC+o(n) \quad  \text{and } \quad \V[X_n(m)] \ = \  o(n^2).
\end{align}
\end{lem}

We define another random variable similarly.

\begin{lem} Let $\{G_n\}$ be a positive linear recurrence sequence, and let $S\subseteq \{G_n\}$ be a set with positive density $d$ in $\{G_n\}$. For $m$ chosen uniformly at random in $[0, G_{n+1})$, let
\begin{align}
Y_n(m) \ := \ \sum_{j\in T_n}{a_j},
\end{align} where  $T_n=\{j\leq n|G_{n+1-j}\in S\}$.
Then, for some constant $C>0$, we have
\begin{align}
\E[Y_n(m)] \ = \ dnC+o(n), \ \ \ \V[Y_n(m)] \ = \ o(n^2).
\end{align}
\end{lem}


\begin{proof}
We first compute the expected value. 
We have 
\begin{align}
\E[Y_n(m)] \ & = \ \E\left[\sum_{j\in T_n}{a_j}\right] \ = \
\sum_{j\in T_n}{\E[a_j]} \ = \ \sum_{j\in T_n}{\sum_{k=1}^{\infty}{kp_{j,k}(n)}} \nonumber \\
& = \ O(\log{n})+\sum_{j\in T_n}{\sum_{k=1}^{\infty}{kp_k(n)(1+o(1))}}. \nonumber \\
& = \ O(\log{n})+dn(1+o(1))\sum_{k=1}^{\infty}{kp_k(n)} \nonumber \\
& = \ O(\log{n})+d(1+o(1))\sum_{j=1}^n{\sum_{k=1}^{\infty}{kp_k(n)}} \nonumber \\
& = \ O(\log{n})+d(1+o(1))\sum_{j=1}^n{\sum_{k=1}^{\infty}{kp_{j,k}(n)}} \nonumber \\
& = \ O(\log{n})+d(1+o(1))\sum_{j=1}^n{\E[a_j]} \ = \ O(\log{n})+\E[X_n(m)]d(1+o(1)) \nonumber \\
& = \ dnC+o(n).
\end{align}
Note that the above sums are actually finite, since $p_{j,k}=p_k=0$ for sufficiently large $k$.  The $\log{n}$ term  appears since Lemma \ref{pk} only allows us to say $p_{j,k}=p_k(1+o(1))$ when $\log{n}<j<n-\log{n}$.
We now must consider the variance.  First note that if $i+\log{n}<j$, then letting
\begin{align}q_{i,r}(n):=\p{\text{the block containing} \ a_i \ \text{ends at} \ a_{i+r}|a_i=k},
\end{align}
we have
\begin{align}
\p{a_j=\l|a_i=k} \ &= \ \sum_{r=1}^{L-1}{q_{i,r}(n)p_{j-i-r,\l}(n)} \nonumber \\
&= \ (1+o(1))p_{\l}(n)\sum_{r=1}^{L-1}{q_{i,r}(n)} \nonumber \\
&= \ p_{\l}(n)(1+o(1)).
\end{align}
Thus, we compute
\begin{align}
\E[Y_n(m)^2] \ & = \ \E\left[\sum_{i,j\in T_n}{a_ia_j}\right] \  = \
\sum_{i,j\in T_n}{\E[a_ia_j]} \nonumber \\
& = \ \sum_{i,j\in T_n}{\sum_{k,\l=1}^{\infty}{k\l p_{i,k}(n)\p{a_j=\l|a_i=k}}} \nonumber \\
& = \ O(n\log{n})+2\sum_{\substack{i,j\in T_n \\ 2\log{n}<i+\log{n}<j<n-\log{n}}}{\sum_{k,\l=1}^{\infty}{k\l p_{i,k}(n)\p{a_j=\l|a_i=k}}} \nonumber \\
& \leq \ O(n\log{n})+2\sum_{\substack{i,j\in T_n \\ 2\log{n}<i+\log{n}<j<n-\log{n}}}{\sum_{k,\l=1}^{\infty}{k\l p_k(n)p_{\l}(n)(1+o(1))}} \nonumber \\
& = \ O(n\log{n})+(1+o(1))d^2n^2\sum_{k,\l=1}^{\infty}{k\l p_k(n)p_{\l}(n)} \nonumber \\
& = \ O(n\log{n})+(1+o(1))d^2n^2\left(\sum_{k=1}^{\infty}{k p_k(n)}\right)^2 \nonumber \\
& = \ O(n\log{n})+d^2n^2C^2(1+o(1)) \ = \ d^2n^2C^2+o(n^2).
\end{align}
Thus
\begin{align}
\V[Y_n(m)]\ = \ \E[Y_n(m)^2]-\E[Y_n(m)]^2 \ = \ o(n^2),
\end{align} completing the proof.
\end{proof}


We are now ready to prove our main result. The idea of the proof is that the results above strongly concentrate $Y_n(m)$ and $X_n(m)$.


\begin{proof}[Proof of Theorem \ref{thm:mainresult}]
Note that the proportion of the summands in $m$'s Zeckendorf decomposition which belong to $S$ is $\frac{Y_n(m)}{X_n(m)}$, where $X_n(m),Y_n(m)$ are defined as in the previous lemmas. Therefore it suffices to show that for any $\e>0$, with probability $1+o(1)$ we have
\begin{align}
\left|\frac{Y_n(m)}{X_n(m)}-d\right|\ < \ \e.
\end{align}
By Chebyshev's inequality, letting $g(n)=n^{1/2}\V[X_n(m)]^{-1/4}$, we obtain
\begin{align}
\p{|X_n(m)-\E[X_n(m)]|>\frac{\E[X_n(m)]}{g(n)}} \ \leq\  \frac{\V[X_n(m)]g(n)^2}{\E[X_n(m)]^2}\ = \ o(1).
\end{align}
Letting
\begin{align}
e_1(n) \ := \ \frac{1}{nC}\left(\frac{\E[X_n(m)]}{g(n)}+\left|\E[X_n(m)]-nC\right|\right),
\end{align}
we have with probability $1+o(1)$ that
\begin{align}
nC(1-e_1(n))\ \leq \ X_n(m) \ \leq \ Cn(1+e_1(n)).
\end{align}
Note that $e_1(n)=o(1)$.  A similar argument for $Y_n(m)$ shows that there exists some $e_2(n)=o(1)$ such that with probability $1+o(1)$ we have
\be
dnC(1-e_2(n))\ \leq \ Y_n(m) \ \leq \ dnC(1+e_2(n)).
\ee
Therefore, we have that
\begin{align}
\frac{Y_n(m)}{X_n(m)}\ \leq\ \frac{dnC(1+e_2(n))}{nC(1-e_1(n))}<d+\e,
\end{align}
with probability $1+o(1)$, and we can similarly obtain
\begin{align}
\frac{Y_n(m)}{X_n(m)} \ > \  d-\e.
\end{align}
Thus we conclude that with probability $1+o(1)$
\begin{align}
\left|\frac{Y_n(m)}{X_n(m)}-d\right| \ < \ \e,
\end{align}
completing the proof.
\end{proof}


\section{Conclusion and Future Work}

We were able to handle the behavior of Zeckendorf decompositions in fairly general settings by cleverly separating any decomposition into manageable blocks. The key step was the notion of a super-legal decomposition, which simplified the combinatorial analysis of the generalized Zeckendorf decompositions significantly. This allowed us to prove not just Benford behavior for the leading digits, but also similar results for other sequences with positive density.

We obtained results for a large class of linear recurrences by considering only the main term of Binet's formula for each linear recurrence. In future work we plan on revisiting these problems for other sequences. Obvious candidates include far-difference representations \cite{Al,DDKMV}, $f$-decompositions \cite{DDKMMV}, and recurrences with leading term zero (some of which do not have unique decompositions) \cite{CFHMN1,CFHMN2}.

\ \\

\end{document}